\documentclass{article}

\catcode`\@=11

\thicklines
\newskip\Einheit \Einheit=.6cm
\newcount\xcoord \newcount\ycoord
\newdimen\xdim \newdimen\ydim \newdimen\PfadD@cke \newdimen\Pfadd@cke
\def\PfadDicke#1{\PfadD@cke#1 \divide\PfadD@cke by2
\Pfadd@cke\PfadD@cke \multiply\PfadD@cke by2}
\long\def\LOOP#1\REPEAT{\def\BODY{#1}\ITERATE}
\def\ITERATE{\BODY \let\next\ITERATE \else\let\next\relax\fi \next}
\let\REPEAT=\fi
\def\Punkt{\hbox{\raise-2pt\hbox to0pt{\hss\scriptsize$\bullet$\hss}}}

\def\DuennPunkt(#1,#2){\unskip
  \raise#2 \Einheit\hbox to0pt{\hskip#1 \Einheit
          \raise-1.5pt\hbox to0pt{\hss\tiny$\bullet$\hss}\hss}}

\def\NormalPunkt(#1,#2){\unskip
  \raise#2 \Einheit\hbox to0pt{\hskip#1 \Einheit
          \raise-3pt\hbox to0pt{\hss\large$\bullet$\hss}\hss}}
\def\DickPunkt(#1,#2){\unskip
  \raise#2 \Einheit\hbox to0pt{\hskip#1 \Einheit
          \raise-4pt\hbox to0pt{\hss\Large$\bullet$\hss}\hss}}
\def\Kreis(#1,#2){\unskip
  \raise#2 \Einheit\hbox to0pt{\hskip#1 \Einheit
          \raise-4pt\hbox to0pt{\hss\Large$\circ$\hss}\hss}}
\def\Diagonale(#1,#2)#3{\unskip\leavevmode
  \xcoord#1\relax \ycoord#2\relax
      \raise\ycoord \Einheit\hbox to0pt{\hskip\xcoord \Einheit
         \unitlength\Einheit
         \line(1,1){#3}\hss}}
\def\AntiDiagonale(#1,#2)#3{\unskip\leavevmode
  \xcoord#1\relax \ycoord#2\relax \advance\xcoord by -0.05\relax
      \raise\ycoord \Einheit\hbox to0pt{\hskip\xcoord \Einheit
         \unitlength\Einheit
         \line(1,-1){#3}\hss}}
\def\Pfad(#1,#2),#3\endPfad{\unskip\leavevmode
  \xcoord#1 \ycoord#2 \thicklines\ZeichnePfad#3\endPfad\thinlines}
\def\ZeichnePfad#1{\ifx#1\endPfad\let\next\relax
  \else\let\next\ZeichnePfad
    \ifnum#1=1
      \raise\ycoord \Einheit\hbox to0pt{\hskip\xcoord \Einheit
         \vrule height\Pfadd@cke width1 \Einheit depth\Pfadd@cke\hss}%
      \advance\xcoord by 1
     \else\ifnum#1=2
      \raise\ycoord \Einheit\hbox to0pt{\hskip\xcoord \Einheit
         \unitlength\Einheit
         \line(0,1){1}\hss}
      \advance\xcoord by 0
      \advance\ycoord by 1
 \else\ifnum#1=3
      \raise\ycoord \Einheit\hbox to0pt{\hskip\xcoord \Einheit
         \unitlength\Einheit
         \line(1,1){1}\hss}
      \advance\xcoord by 1
      \advance\ycoord by 1
    \else\ifnum#1=4
      \raise\ycoord \Einheit\hbox to0pt{\hskip\xcoord \Einheit
         \unitlength\Einheit
         \line(1,-1){1}\hss}
      \advance\xcoord by 1
      \advance\ycoord by -1
   \else\ifnum#1=5
      \raise\ycoord \Einheit\hbox to0pt{\hskip\xcoord \Einheit
         \unitlength\Einheit
         \line(2,1){2}\hss}
      \advance\xcoord by 2
      \advance\ycoord by 1
	  \else\ifnum#1=6
      \raise\ycoord \Einheit\hbox to0pt{\hskip\xcoord \Einheit
         \unitlength\Einheit
         \line(2,-1){2}\hss}
      \advance\xcoord by 2
      \advance\ycoord by -1
	  \else\ifnum#1=7
      \raise\ycoord \Einheit\hbox to0pt{\hskip\xcoord \Einheit
         \unitlength\Einheit
         \line(3,1){3}\hss}
      \advance\xcoord by 3
      \advance\ycoord by 1
	  \else\ifnum#1=8
      \raise\ycoord \Einheit\hbox to0pt{\hskip\xcoord \Einheit
         \unitlength\Einheit
         \line(3,-1){3}\hss}
      \advance\xcoord by 3
      \advance\ycoord by -1
    \fi\fi\fi\fi\fi\fi\fi\fi
  \fi\next}
\def\hSSchritt{\leavevmode\raise-.4pt\hbox
to0pt{\hss.\hss}\hskip.2\Einheit
  \raise-.4pt\hbox to0pt{\hss.\hss}\hskip.2\Einheit
  \raise-.4pt\hbox to0pt{\hss.\hss}\hskip.2\Einheit
  \raise-.4pt\hbox to0pt{\hss.\hss}\hskip.2\Einheit
  \raise-.4pt\hbox to0pt{\hss.\hss}\hskip.2\Einheit}
\def\vSSchritt{\vbox{\baselineskip.2\Einheit\lineskiplimit0pt
\hbox{.}\hbox{.}\hbox{.}\hbox{.}\hbox{.}}}
\def\DSSchritt{\leavevmode\raise-.4pt\hbox to0pt{%
  \hbox to0pt{\hss.\hss}\hskip.2\Einheit
  \raise.2\Einheit\hbox to0pt{\hss.\hss}\hskip.2\Einheit
  \raise.4\Einheit\hbox to0pt{\hss.\hss}\hskip.2\Einheit
  \raise.6\Einheit\hbox to0pt{\hss.\hss}\hskip.2\Einheit
  \raise.8\Einheit\hbox to0pt{\hss.\hss}\hss}}
\def\dSSchritt{\leavevmode\raise-.4pt\hbox to0pt{%
  \hbox to0pt{\hss.\hss}\hskip.2\Einheit
  \raise-.2\Einheit\hbox to0pt{\hss.\hss}\hskip.2\Einheit
  \raise-.4\Einheit\hbox to0pt{\hss.\hss}\hskip.2\Einheit
  \raise-.6\Einheit\hbox to0pt{\hss.\hss}\hskip.2\Einheit
  \raise-.8\Einheit\hbox to0pt{\hss.\hss}\hss}}
\def\SPfad(#1,#2),#3\endSPfad{\unskip\leavevmode
  \xcoord#1 \ycoord#2 \ZeichneSPfad#3\endSPfad}
\def\ZeichneSPfad#1{\ifx#1\endSPfad\let\next\relax
  \else\let\next\ZeichneSPfad
    \ifnum#1=1
      \raise\ycoord \Einheit\hbox to0pt{\hskip\xcoord \Einheit
         \hSSchritt\hss}%
      \advance\xcoord by 1
    \else\ifnum#1=2
      \raise\ycoord \Einheit\hbox to0pt{\hskip\xcoord \Einheit
        \hbox{\hskip-2pt \vSSchritt}\hss}%
      \advance\ycoord by 1
    \else\ifnum#1=3
      \raise\ycoord \Einheit\hbox to0pt{\hskip\xcoord \Einheit
         \DSSchritt\hss}
      \advance\xcoord by 1
      \advance\ycoord by 1
    \else\ifnum#1=4
      \raise\ycoord \Einheit\hbox to0pt{\hskip\xcoord \Einheit
         \dSSchritt\hss}
      \advance\xcoord by 1
      \advance\ycoord by -1
    \fi\fi\fi\fi
  \fi\next}
\def\Koordinatenachsen(#1,#2){\unskip
 \hbox to0pt{\hskip-.5pt\vrule height#2 \Einheit width.5pt depth1
\Einheit}%
 \hbox to0pt{\hskip-1 \Einheit \xcoord#1 \advance\xcoord by1
    \vrule height0.25pt width\xcoord \Einheit depth0.25pt\hss}}
\def\Koordinatenachsen(#1,#2)(#3,#4){\unskip
 \hbox to0pt{\hskip-.5pt \ycoord-#4 \advance\ycoord by1
    \vrule height#2 \Einheit width.5pt depth\ycoord \Einheit}%
 \hbox to0pt{\hskip-1 \Einheit \hskip#3\Einheit
    \xcoord#1 \advance\xcoord by1 \advance\xcoord by-#3
    \vrule height0.25pt width\xcoord \Einheit depth0.25pt\hss}}
\def\Gitter(#1,#2){\unskip \xcoord0 \ycoord0 \leavevmode
  \LOOP\ifnum\ycoord<#2
    \loop\ifnum\xcoord<#1
      \raise\ycoord \Einheit\hbox to0pt{\hskip\xcoord
\Einheit\Punkt\hss}%
      \advance\xcoord by1
    \repeat
    \xcoord0
    \advance\ycoord by1
  \REPEAT}
\def\Gitter(#1,#2)(#3,#4){\unskip \xcoord#3 \ycoord#4 \leavevmode
  \LOOP\ifnum\ycoord<#2
    \loop\ifnum\xcoord<#1
      \raise\ycoord \Einheit\hbox to0pt{\hskip\xcoord
\Einheit\Punkt\hss}%
      \advance\xcoord by1
    \repeat
    \xcoord#3
    \advance\ycoord by1
  \REPEAT}
\def\Label#1#2(#3,#4){\unskip \xdim#3 \Einheit \ydim#4 \Einheit
  \def\lo{\advance\xdim by-.5 \Einheit \advance\ydim by.5 \Einheit}%
  \def\llo{\advance\xdim by-.25cm \advance\ydim by.5 \Einheit}%
  \def\loo{\advance\xdim by-.5 \Einheit \advance\ydim by.25cm}%
  \def\o{\advance\ydim by.25cm}%
  \def\ro{\advance\xdim by.5 \Einheit \advance\ydim by.5 \Einheit}%
  \def\rro{\advance\xdim by.25cm \advance\ydim by.5 \Einheit}%
  \def\roo{\advance\xdim by.5 \Einheit \advance\ydim by.25cm}%
  \def\l{\advance\xdim by-.30cm}%
  \def\r{\advance\xdim by.30cm}%
  \def\lu{\advance\xdim by-.5 \Einheit \advance\ydim by-.6 \Einheit}%
  \def\llu{\advance\xdim by-.25cm \advance\ydim by-.6 \Einheit}%
  \def\luu{\advance\xdim by-.5 \Einheit \advance\ydim by-.30cm}%
  \def\u{\advance\ydim by-.30cm}%
  \def\ru{\advance\xdim by.5 \Einheit \advance\ydim by-.6 \Einheit}%
  \def\rru{\advance\xdim by.25cm \advance\ydim by-.6 \Einheit}%
  \def\ruu{\advance\xdim by.5 \Einheit \advance\ydim by-.30cm}%
  #1\raise\ydim\hbox to0pt{\hskip\xdim
     \vbox to0pt{\vss\hbox to0pt{\hss$#2$\hss}\vss}\hss}%
}
\catcode`\@=12

\usepackage{amssymb}
\usepackage{latexsym}
\usepackage{amsmath}
\usepackage{amsthm,pstricks}
\usepackage{xspace}
\topskip  \theoremstyle{remark}
\theoremstyle{plain}
\newtheorem{theorem}{Theorem}[section]
\newtheorem{lemma}[theorem]{Lemma}

\newtheorem{proposition}[theorem]{Proposition}

\newcommand{\asc}{\text{asc}}
\newcommand{\ase}{\text{ase}}
\def\gl{ground level\xspace}

\begin{document}

\title{Restricted ascent sequences and Catalan numbers}
\author{David Callan\\
\small Department of Statistics, University of Wisconsin, Madison, WI 53706\\[-0.8ex]
\small\texttt{callan@stat.wisc.edu}\\[1.8ex]
Toufik Mansour\\
\small Department of Mathematics, University of Haifa, 31905 Haifa, Israel\\[-0.8ex]
\small\texttt{tmansour@univ.haifa.ac.il}\\[1.8ex]
Mark Shattuck\\
\small Department of Mathematics, University of Tennessee, Knoxville, TN 37996\\[-0.8ex]
\small\texttt{shattuck@math.utk.edu}\\[1.8ex]
}

\date{\small }
\maketitle

\begin{abstract}
Ascent sequences are those consisting of non-negative integers in which the size of each letter is restricted by the number of ascents preceding it and have been shown to be equinumerous with the (2+2)-free posets of the same size.  Furthermore, connections to a variety of other combinatorial structures, including set partitions, permutations, and certain integer matrices, have been made.  In this paper, we identify all members of the (4,4)-Wilf equivalence class for ascent sequences corresponding to the Catalan number $C_n=\frac{1}{n+1}\binom{2n}{n}$. This extends recent work concerning avoidance of a single pattern and provides apparently new combinatorial interpretations for $C_n$.  In several cases, the subset of the class consisting of those members having exactly $m$ ascents is given by the Narayana number $N_{n,m+1}=\frac{1}{n}\binom{n}{m+1}\binom{n}{m}$.
\end{abstract}



\section{Introduction}

An \emph{ascent} in a sequence $x_1x_2\cdots x_k$ is a place $j \geq 1$ such that $x_j<x_{j+1}$.  An \emph{ascent sequence} $x_1x_2\cdots x_n$ is one consisting of non-negative integers in which $x_1=0$ and satisfying the condition,
$$x_i\leq\asc(x_1x_2\cdots x_{i-1})+1, \qquad 1 < i \leq n,$$
where $\asc(x_1x_2\cdots x_k)$ is the number of ascents in the sequence $x_1x_2\cdots x_k$.  An example of such a sequence is $0101303543$, whereas $00110242$ is not since $4$ exceeds $\asc(001102)+1=3$.  Ascent sequences were first studied by Bousquet-M{\'e}lou, Claesson, Dukes, and Kitaev \cite{BCD}, where they were shown to have the same cardinality as the (2+2)-free posets of the same size.  Since then they have been studied in a series of papers by various authors and connections have been made to many other combinatorial structures.  See, for example, \cite{DP,DRS,KR} as well as \cite[Section 3.2.2]{K} for further information.

We will refer to a sequence of non-negative integers, where repetitions are allowed, as a \emph{pattern}.  Let $\pi=\pi_1\pi_2\cdots \pi_n$ be an ascent sequence and $\tau=\tau_1\tau_2\cdots\tau_m$ be a pattern.  Then we say that $\pi$ \emph{contains} $\tau$ if $\pi$ has a subsequence that is order isomorphic to $\tau$, that is, if there is a subsequence $\pi_{f(1)},\pi_{f(2)},\ldots,\pi_{f(m)}$, where $1\leq f(1)<f(2)<\cdots<f(m)\leq n$, such that for all $1 \leq i,j \leq m$, we have $\pi_{f(i)}<\pi_{f(j)}$ if and only if $\tau_i<\tau_j$ and $\pi_{f(i)}>\pi_{f(j)}$ if and only if $\tau_i>\tau_j$.  Otherwise, $\pi$ is said to \emph{avoid} $\tau$.  For example, the ascent sequence $01001341404654$ has three occurrences of the pattern $210$, namely, the subsequences $310$, $410$, and $654$, but avoids the pattern $201$.  Note that within an occurrence of a pattern $\tau$, letters corresponding to equal letters in $\tau$ must be equal within the occurrence.

To be consistent with the usual notation for ascent sequences which contains $0$s, we will write patterns for ascent sequences using non-negative integers in accordance with \cite{DS}, though patterns for other structures like permutations and set partitions are traditionally written with positive integers.  Given a set of patterns $T$, let $\mathcal{S}_n(T)$ denote the set of ascent sequences of length $n$ avoiding all of the patterns in $T$, and let $S_n(T)$ denote the number of such sequences.  Furthermore, if $0 \leq m <n$, then let $\mathcal{S}_{n,m}(T)$ denote the subset of $\mathcal{S}_n(T)$ whose members contain exactly $m$ ascents and let $S_{n,m}(T)=|\mathcal{S}_{n,m}(T)|$.

The Catalan numbers $C_n=\frac{1}{n+1}\binom{2n}{n}$ have been shown to count many structures in both enumerative and algebraic combinatorics.  Perhaps the most fundamental structure enumerated by $C_n$ is the set of lattice paths from $(0,0)$ to $(2n,0)$ using up $(1,1)$ steps and down $(1,-1)$ steps that never dip below the $x$-axis (called \emph{Catalan} or \emph{Dyck} paths).  In terms of avoidance, it was shown by Knuth \cite{Kn} that $C_n$ counts the permutations of $[n]=\{1,2,\ldots,n\}$ avoiding a single classical pattern $\tau$, where $\tau$ is any member of $S_3$ (see also \cite[Chapter 4]{K}).  Later, it was shown that $C_n$ is also the number of partitions of $[n]$ avoiding either $1212$ or $1221$ (called \emph{non-crossing} and \emph{non-nesting} partitions, respectively; see, e.g., \cite{Kl}).  More recently, in terms of ascent sequences, the numbers $C_n$ have been shown to count $\mathcal{S}_n(021)$, see \cite{DS}.  See \cite[A000108]{Sl} for more information!
  on these numbers.  To date, there are at least 204 structures known to be enumerated by the Catalan numbers; see Stanley's website \cite{St} for a complete list.

Recall that the Narayana numbers given by $N_{n,m}=\frac{1}{n}\binom{n}{m}\binom{n}{m-1}$, $1 \leq m \leq n$, refine the Catalan numbers in that $C_n=\sum_{m=1}^n N_{n,m}$ and have generating function
$$\sum_{1 \leq m \leq n}N_{n,m}x^ny^m=\frac{1-x(1+y)-\sqrt{(1+x(1-y))^2-4x}}{2x}.$$
Among other things, the Narayana numbers $N_{n,m}$ count the Dyck paths of semilength $n$ having $m$ peaks as well as the $132$-avoiding permutations of length $n$ having $m-1$ ascents (see, e.g., \cite[A001263]{Sl} and \cite{Br} for other combinatorial interpretations). More recently, the Narayana numbers have been shown to count certain classes of ascent sequences (see \cite{DS,MS}).

Our main result, Theorem \ref{t1} below, identifies all members of the $(4,4)$-Wilf equivalence class for ascent sequences corresponding to the Catalan number $C_n$ (excepting those pattern pairs that are trivially equivalent to either 0101 or 0012).  In addition to providing apparently new combinatorial interpretations for the Catalan (and Narayana) sequence, this extends recent work concerning the avoidance of a single pattern in ascent sequences (see \cite{CD,DS,MS,Y}).  We now state our main result.

\begin{theorem}\label{t1}
If $n \geq 1$, then $S_{n}(u,v)=C_n$ for the following pairs $(u,v)$:
\begin{center}
\begin{tabular}{llllll}
  $(1)~(0001,1012)$ & $(2)~(0010,0021)$ & $(3)~(0011,0021)$\\
  $(4)~(0021,0121)$ & $(5)~(0121,0132)$ & $(6)~(0121,1032)$\\
  $(7)~(0122,0132)$ & $(8)~(0122,1032).$ &    &\\
  \end{tabular}
\end{center}
In addition, we have $S_{n,m}(u,v)=N_{n,m+1}$ if $0 \leq m<n$ for the pairs $(4)$-$(8)$.
\end{theorem}

Theorem \ref{t1} will follow from combining results in the next two sections.  For cases (4)-(8), we show that the ascent sequences avoiding the pair of patterns in question either are in one-to-one correspondence with or synonymous with ascent sequences avoiding the single pattern $021$. We then show that cases (2) and (3) are equivalent and in the third section identify a bijection between (3) and Dyck paths of twice the length.  In the final section, we further enumerate the members of $\mathcal{S}_n(021)$ according to the joint distribution of three statistics and make use of the \emph{kernel method} \cite{BBD} in our derivation.

\section{Pattern avoiding ascent sequences}

We first consider cases (4) and (5) in Theorem \ref{t1} above.

\begin{proposition}\label{p1}
If $n \geq1$, then $\mathcal{S}_{n}(0021,0121)=\mathcal{S}_{n}(0121,0132)=\mathcal{S}_{n}(021)$ and, for $0 \leq m <n$, we have $S_{n,m}(0021,0121)=S_{n,m}(0121,0132)=N_{n,m+1}$.
\end{proposition}
\begin{proof}
We prove only the first set of equalities, the second following from the first by \cite[Theorem 2.15]{DS}.  Clearly, a member of $\mathcal{S}_{n}(021)$ avoids both $0021$ and $0121$.  So we will show that any $\pi \in \mathcal{S}_{n}(0021,0121)$ must avoid $021$
by showing that within $\pi$, no letter $i \geq 1$ can occur to the right of any letter $j>i$.  Suppose, to the contrary, that there is such an occurrence within $\pi$ involving the letters $a$ and $b$, where $b>a$.  Clearly, we must have $a>1$, for if $a=1$, then there would be an occurrence of $0121$ within $\pi$ corresponding to a subsequence of the form $01b1$.  Furthermore, no $a$'s can occur to the left of $b$, for otherwise there would be an occurrence of $0121$ of the form $0aba$.  We may also assume, without loss of generality, that $b$ is left-most of all letters of its kind.  Thus, some letter $c$, where $0<c<a$, must have been repeated prior to the left-most occurrence of $b$ in order to allow for the letter $a$ to be ``skipped'' when one considers a left-to-right scan of the left-most occurrences of the distinct letters of $\pi$. But then there would be an occurrence of $0021$ in $\pi$ corresponding to a subsequence of the form $ccba$.  Thus, there are no such l!
 etters $a$ and $b$ within $\pi$, which completes the proof.

To show $\mathcal{S}_{n}(0121,0132)=\mathcal{S}_{n}(021)$, note first that a member of $\mathcal{S}_{n}(021)$ avoids both $0121$ and $0132$.  To complete the proof, we'll show that if an ascent sequence $\pi=\pi_1\pi_2\cdots \pi_n$ contains $021$, then it must contain $0121$ or $0132$. Suppose that the letters $2$ and $1$ within an occurrence of $021$ in $\pi$ correspond to the actual letters $x$ and $y$, where $x > y \geq 1$; clearly one may take the $0$ within the occurrence to be $\pi_1=0$.  If $y=1$, then $\pi$ would contain an occurrence of $0121$ since there must be a $1$ to the left of $x$ as the first nonzero element in any ascent sequence is $1$.  If $y>1$, then $01xy$ is an occurrence of $0132$ in $\pi$, which completes the proof.
\end{proof}

The following result shows several of the equivalences in Theorem \ref{t1}.

\begin{proposition}\label{p2}
If $n \geq 1$ and $0 \leq m <n$, then
$$S_{n,m}(0121,0132)=S_{n,m}(0121,1032)=S_{n,m}(0122,0132)=S_{n,m}(0122,1032).$$
\end{proposition}
\begin{proof}
We first define a bijection between $\mathcal{S}_{n,m}(0121,0132)$ and $\mathcal{S}_{n,m}(0122,0132)$.  To do so, note that $\rho \in \mathcal{S}_{n,m}(0121,0132)$ implies either $\rho$ is binary or of the form $\rho=\alpha\beta$, where $\alpha$ is a binary word starting with $0$ but not consisting of all $0$s and $\beta$ is a nonempty word not containing $1$ and containing no descent in which the smaller letter is greater than $1$.  Define the ascent sequence $\rho'$ by either letting $\rho'=\rho$ if $\rho$ is binary or otherwise letting $\rho'=\alpha\beta'$, where $\beta'$ is obtained from $\beta$ by replacing all but the first occurrence of each letter $r>1$ with $1$. For example, if $n=12$, $m=5$, and $\rho=001122020404\in\mathcal{S}_{12,5}(0121,0132)$, then $\rho'=001121010401\in \mathcal{S}_{12,5}(0122,0132)$.  One may verify that the mapping $\rho\mapsto \rho'$ is a bijection from $\mathcal{S}_{n,m}(0121,0132)$ to $\mathcal{S}_{n,m}(0122,0132)$.  A similar bijection s!
 hows $S_{n,m}(0121,1032)=S_{n,m}(0122,1032)$.

To complete the proof, we show $S_{n,m}(0121,0132)=S_{n,m}(0121,1032)$, and for this, we show in fact $\mathcal{S}_{n}(0121,0132)=\mathcal{S}_{n}(0121,1032)=\mathcal{S}_{n}(0121,0132,1032)$. To do so,  suppose $\lambda \in \mathcal{S}_{n}(0121)$.  If $\lambda$ contains an occurrence of $1032$ corresponding to a subsequence $abcd$, then it would also contain an occurrence of $0132$ corresponding to the subsequence $0acd$.  Suppose then that $\lambda$ contains an occurrence $\tau$ of $0132$ corresponding to a subsequence $abcd$, where we may assume $a=0$, $b=1$, and that the occurrence of the letter $c$ is left-most.  Since $\lambda$ avoids $0121$, no letter $d$ can occur to the left of the first occurrence of $c$.  Thus, there must be at least one descent to the left of the left-most occurrence of any letter $i>d$, in particular, for $i=c$.  Then $\lambda$ would contain an occurrence of $1032$ corresponding to the subsequence $xycd$, where $x$ and $y$ denote the letters belon!
 ging to the aforementioned descent.
\end{proof}

Our next result shows the equivalence of the cases of avoiding $\{0010,0021\}$ and $\{0011,0021\}$.

\begin{proposition}\label{p4}
If $n \geq 1$ and $0 \leq m < n$, then $S_{n,m}(0010,0021)=S_{n,m}(0011,0021)$.
\end{proposition}
\begin{proof}
We will show that both sets can be generated inductively by equivalent recursive procedures.  We first consider the case of avoiding $0010$ and $0021$.  Let $\mathcal{U}_n=\mathcal{S}_n(0010,0021)$.  Let $\pi \in \mathcal{U}_\ell$ for some $\ell<n$, where we use the alphabet of positive rather than non-negative integers.  We wish to add letters to $\pi$, including at least one zero, so as to  create a member of $\mathcal{U}_n$ (with the typical representation).  If $\pi \neq 12\cdots \ell$, then we may write $\pi=\alpha\beta$, where $\alpha$ is possibly empty, $t \in \{2,3,\ldots,\ell\}$, and $\beta=\beta_1\beta_2\cdots \beta_t$ with $\beta_1\geq \beta_2$ and $\beta_2<\beta_3<\cdots<\beta_t$.  In this case, we will say that $\pi$ has exactly $t$ \emph{active sites} corresponding to the slots between the letters $\beta_i$ and $\beta_{i+1}$ for each $i \in [t-1]$ as well as the position directly after letter $\beta_t$, in which case we will write $\ase(\pi)=t$.  If $\pi=12\cdots\ell$, then we let $\ase(\pi)=\ell+1$, the active sites in this case corresponding to the positions directly following each letter in the word $01\cdots \ell$.  Let $\mathcal{U}_{\ell,t}$ denote the subset of $\mathcal{U}_{\ell}$ consisting of those members for which $\ase(\pi)=t$.

We now generate members of $\mathcal{U}_n$ from members $\pi \in \mathcal{U}_{\ell,t}$ for various $\ell<n$ and $t$ by writing a single $0$ before $\pi$ and then inserting letters in the active sites of $\pi$ as follows.
\begin{align*}
1.&\text{ Choose some subset, possibly empty, of the active sites of $\pi$ in which to write at least one}\\
&\text{ letter.}\\
2.&\text{ In the left-most active site selected, write a run of $0$s.}\\
3.&\text{ In each subsequently chosen active site, write a string of the letter directly preceding it.}
\end{align*}
In the case when $\pi\neq12\cdots\ell$ and the (left-most) active site selected in the first step above directly follows $\beta_1$, then an additional ascent is created when a run of zeros is written there since $\beta_1 \geq \beta_2\geq 1$.  Let $0\pi'$ denote the ascent sequence in this case that results when the three steps above are performed on $0\pi$.  Thus, one may also add in this case to the end of $0\pi'$ a sequence of the form $(m+1)^{i_1}(m+2)^{i_2}\cdots (m+d)^{i_d}$, where $d\geq1$, $i_j\geq 1$ for each $j \in [d]$, and the number of ascents in the original sequence $\pi$ is $m-2$.

Let $\mathcal{V}_n=\mathcal{S}_n(0011,0021)$ and $\mathcal{V}_{\ell,t}$ be defined exactly as $\mathcal{U}_{\ell,t}$ above except that $0011$ occurs in place of $0010$.  Then one may describe a similar recursive procedure for generating the members of $\mathcal{V}_n$ from the members of $\mathcal{V}_{\ell,t}$ for various $\ell<n$ and $t$, except that step 3 is replaced with

\noindent \text{~~3$'$. In each subsequently chosen active site, write a string of $0$s,}

and the sequence $(m+1)^{i_1}(m+2)^{i_2}\cdots (m+d)^{i_d}$ described in the preceding paragraph is replaced with $(m+1)0^{i_1-1}(m+2)0^{i_2-1}\cdots (m+d)0^{i_d-1}$.

Comparing the cardinalities of $\mathcal{U}_{n,r}$ and $\mathcal{V}_{n,r}$ for $1 \leq n \leq 3$ and $2 \leq r \leq n+1$ shows that they are the same.  Comparing the two procedures above, one sees that the number of members of $\mathcal{U}_{n,r}$ that arise from each member of $\mathcal{U}_{\ell,t}$ is the same as the number of members of $\mathcal{V}_{n,r}$ that arise from each member of $\mathcal{V}_{\ell,t}$ for various $r$, $t$, and $\ell$, where $n-\ell\geq1$ denotes the total number of letters added in either procedure (in the form of $0$s, repeated letters $\beta_j$, or letters $m+i$ at the end).  Thus, by induction, we have $|\mathcal{U}_{n,r}|=|\mathcal{V}_{n,r}|$ for all $n$ and $r$.  Note that the preceding shows further that the sets $\mathcal{U}_{n,r}$ and $\mathcal{V}_{n,r}$ have equivalent refinements according to the number of ascents for fixed $n$ and $r$, which implies the result.
\end{proof}

\textbf{Remark:}  By \cite[Lemma 2.4]{DS}, the case of avoiding $\{0001,1012\}$ corresponds to the avoidance problem on set patterns for the patterns $1112$ and $12123$, which was treated in \cite{MS0}. It is then seen that the number of members of $\mathcal{S}_n(0001,1012)$ containing $m$ distinct letters is given by $N_{n,m}$ for $1 \leq m \leq n$, by \cite[Theorem 1.1]{MS0}.

\textbf{Remark:} From \cite[Lemma 2.4]{DS}, there are several cases of $(4,4)$ for ascent sequences that are logically equivalent to avoiding either $1212$ or $1123$ by set partitions which we do not list here.

\section{A bijection for the case \{0011,0021\}}

A Dyck path is one consisting of up $(1,1)$ steps and down $(1,-1)$ steps which we will denote by $U$ and $D$, respectively, starting from the origin and ending on the $x$-axis. A \emph{return} in a Dyck path is a noninitial vertex on the $x$-axis.  An \emph{elevated}
Dyck path is a nonempty Dyck path whose only return is at the terminal vertex.
The returns of a nonempty Dyck path split it into elevated Dyck paths
called the \emph{components} of the path. Thus a nonempty Dyck path is
elevated when it has precisely one component.

In this section, we will define a bijection $\phi$ from (0011,\:0021)-avoiding ascent
sequences $u$ of size $n$ to Dyck paths of size $n$, where size is measured as
length for ascent sequences and number of upsteps for Dyck paths.
The bijection $\phi $ sends the number of 0s
in $u$ to the number of components in the corresponding Dyck path
and this will serve as an inductive hypothesis in the definition.
We then compare some statistics on (0011,\:0021)-avoiding ascent
sequences with statistics on Dyck paths.

We make now the following useful observation.
Suppose that $u$ avoids the patterns 0011 and 0021 and that a number
$b$ serves as the 1 of a 001 pattern in $u$.
It is plain that the first occurrence of $b$ in that role is also the last
(else a 0011 would be present),
and that the entries after this occurrence and $>b$ are (strictly) increasing
(else a 0021  would be present).
Thus, if 0 occurs more than once, the nonzero entries following the second 0
must be increasing. In particular, 1 occurs at most once after the second 0 and,
when it does occur, the entries between the second 0 and this 1 are all 0.

\subsection{The bijection $\phi$}

First, set $\phi\big((0)\big) = U\!D$.
Now split (0011,\:0021)-avoiding ascent sequences (avoiders, for short) of length $\ge 2$ into 4 classes:
(1) those consisting entirely of 0s,
(2) those with only one 0 (necessarily at the start),
(3) those with more than one 0 but not all 0s and with no 1 after the second 0, and
(4) those with more than one 0 and a single 1 after the second 0.
We deal with each case in turn.
Suppose we are given an avoider $u$ of length $\ge 2$.

Case 1. $u$ is a sequence of $n$ 0s. Set $\phi(u)= (U\!D)^n$, the ``sawtooth'' Dyck path.

Case 2. $u$ has only one 0. Delete the first 0, decrement by 1 all other entries to get a
one-size-smaller avoider $w$, then elevate $\phi(w)$, that is, concatenate $U, \, \phi(w)$, and $D$.

Case 3. $u$ has more than one 0 but is not all 0s and has no 1 after the second 0. Here $u$ must start 01 since, if $u$ starts 00 and is not all 0s, then a 1 must occur after the second 0.

Let $k \ge 2$ denote the total number of 0s in $u$.
Form a one-size-smaller avoider $w$ by deleting the first 0
and decrementing by 1 all other \emph{nonzero} entries. Note that
$w$ still has at least $k$ 0s
and (by the inductive hypothesis) $\phi(w)$ has at least $k$ components.
Let $P$ denote the result of erasing the last $k-1$ components from $\phi(w)$
and let $Q$ denote the subpath consisting of the erased components.
Thus neither $P$ nor $Q$ is the empty path. Set $\phi(u)= UPDQ$, schematically,

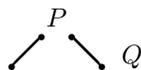
\Einheit=0.4cm
\[
\Pfad(-2,0),3\endPfad
\Pfad(0,1),4\endPfad
\DuennPunkt(-2,0)
\DuennPunkt(-1,1)
\DuennPunkt(0,1)
\DuennPunkt(1,0)
\Label\o{P}(-0.5,1)
\Label\o{Q}(2,-0.3)
\]
a Dyck path with $k$ components. \\

Case 4. $u$ has more than one 0 and a (single) 1 after the second 0. Let $j\ge0$
denote the length of the segment strictly between the second 0 and the single 1.
Delete this segment (it consists entirely of 0s) and the single 1.
Then decrement each \emph{nonzero} entry after the second 0 (in turn) by at least 1
(decrementing by 1 may result in an offending pattern) but by no more than necessary to
avoid an offending pattern. For example, the avoider
0123341120013057 contains 334 and 112 before the second 0, and so 4 and 2 are
certainly forbidden later in the sequence. Thus the segment 013057 after
the second 0 becomes 3057 after deleting and then 1036 after decrementing.

Let $w$ denote the avoider thus obtained. Note that $w$ has length $n-1-j$ and contains
at least two 0s. So, by the inductive hypothesis, $\phi(w)$ is a Dyck path with at
least 2 components and thus of the form $UPDUQDR$ or, schematically,

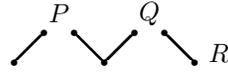
\Einheit=0.4cm
\[
\hspace*{10mm}
\Pfad(-4,0),3\endPfad
\Pfad(-2,1),43\endPfad
\Pfad(1,1),4\endPfad
\DuennPunkt(-4,0)
\DuennPunkt(-3,1)
\DuennPunkt(-2,1)
\DuennPunkt(-1,0)
\DuennPunkt(0,1)
\DuennPunkt(1,1)
\DuennPunkt(2,0)
\Label\o{P}(-2.5,1)
\Label\o{Q}(0.5,1)
\Label\o{R}(2.8,-0.3)
\]

with $P,Q,R$ all Dyck paths (possibly empty).
Set $\phi(u)=(U\!D)^{j+1}UUPDQDR$, schematically,
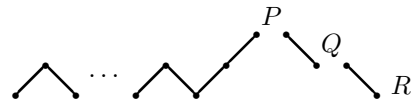
\Einheit=0.4cm
\[
\hspace*{10mm}
\Pfad(-8,0),34\endPfad
\Pfad(-4,0),3433\endPfad
\Pfad(1,2),4\endPfad
\Pfad(3,1),4\endPfad
\DuennPunkt(-8,0)
\DuennPunkt(-7,1)
\DuennPunkt(-6,0)
\DuennPunkt(-4,0)
\DuennPunkt(-3,1)
\DuennPunkt(-2,0)
\DuennPunkt(-1,1)
\DuennPunkt(0,2)
\DuennPunkt(1,2)
\DuennPunkt(2,1)
\DuennPunkt(3,1)
\DuennPunkt(4,0)
\Label\o{\dots}(-5,0)
\Label\o{P}(0.5,2)
\Label\o{Q}(2.5,1)
\Label\o{R}(4.8,-0.3)
\]

with $j+1$ $U\!D$s at the start.  \qed

Thus, for example, $\phi(\,(0,1)\,)=UUDD$ and the effect of $\phi$ on the 5 avoiders of length 3 is shown below (with case number in parentheses).

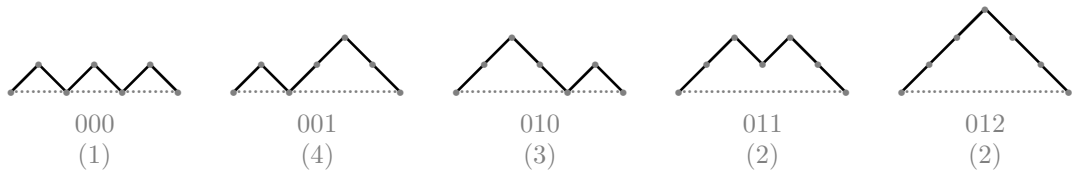
\Einheit=0.37cm
\[
\Pfad(-19,0),343434\endPfad
\Pfad(-11,0),343344\endPfad
\Pfad(-3,0),334434\endPfad
\Pfad(5,0),334344\endPfad
\Pfad(13,0),333444\endPfad
\gray{\SPfad(-19,0),111111\endSPfad
\SPfad(-11,0),111111\endSPfad
\SPfad(-3,0),111111\endSPfad
\SPfad(5,0),111111\endSPfad
\SPfad(13,0),111111\endSPfad}
\DuennPunkt(-19,0)
\DuennPunkt(-18,1)
\DuennPunkt(-17,0)
\DuennPunkt(-16,1)
\DuennPunkt(-15,0)
\DuennPunkt(-14,1)
\DuennPunkt(-13,0)
\DuennPunkt(-11,0)
\DuennPunkt(-10,1)
\DuennPunkt(-9,0)
\DuennPunkt(-8,1)
\DuennPunkt(-7,2)
\DuennPunkt(-6,1)
\DuennPunkt(-5,0)
\DuennPunkt(-3,0)
\DuennPunkt(-2,1)
\DuennPunkt(-1,2)
\DuennPunkt(0,1)
\DuennPunkt(1,0)
\DuennPunkt(2,1)
\DuennPunkt(3,0)
\DuennPunkt(5,0)
\DuennPunkt(6,1)
\DuennPunkt(7,2)
\DuennPunkt(8,1)
\DuennPunkt(9,2)
\DuennPunkt(10,1)
\DuennPunkt(11,0)
\DuennPunkt(13,0)
\DuennPunkt(14,1)
\DuennPunkt(15,2)
\DuennPunkt(16,3)
\DuennPunkt(17,2)
\DuennPunkt(18,1)
\DuennPunkt(19,0)
\Label\o{000}(-16,-1.8)
\Label\o{(1)}(-16,-3)
\Label\o{001}(-8,-1.8)
\Label\o{(4)}(-8,-3)
\Label\o{010}(0,-1.8)
\Label\o{(3)}(0,-3)
\Label\o{011}(8,-1.8)
\Label\o{(2)}(8,-3)
\Label\o{012}(16,-1.8)
\Label\o{(2)}(16,-3)
\]

The inductive step for the assertion ``$\phi$ sends \# 0s to \# components'' clearly holds for cases 1 and 2, and for case 3 by the remarks made in that case. As for case 4, suppose $u$ has a total of $k$ 0s. Then $w$ has precisely $k-j$ 0s. So $\phi(w)$ has $k-j$ components (by the inductive hypothesis) and $R$ has $k-j-2$ components. The final result consists of $j+1$ $U\!D$s followed by an elevated path and then by $R$ for a total of $(j+1)+1+(k-j-2) = k$ components, and the induction step is verified.

To reverse the map, the image paths in the four classes can be distinguished as follows:
(1) sawtooth paths, $(U\!D)^n$, (2) elevated Dyck paths, (3) paths that start $UU$ but are
not elevated, and (4) paths that start $U\!D$ but are not sawtooth. We leave the reader to take it from there.

\subsection{Some equidistributions} \label{some}

\vspace*{-3mm}

The effect of the bijection $\phi$ on some statistics is shown in Table 1 below.
The all-0s ascent sequence $0^n$ corresponds to the sawtooth path $(U\!D)^n$.
To avoid exceptional cases in the definition of some statistics, the all 0s sequence
and the sawtooth path are excluded in Table 1. The abbreviation  ht($P$)
denotes the height (height of highest peak) of the Dyck path $P$.

\begin{center}
  \begin{tabular}{@{} ccc @{}}
    \hline
 (0011,0021)-avoiders, not all 0s &  & Dyck paths, not a sawtooth   \\[.5mm]
    \hline
    & & \\[-3mm]
   \# initial 0s & $\leftrightarrow$ & $
\begin{cases}
    1, & \text{if $P$ starts $U\!D$ and ht($P)\ge 3 $;} \\
    1 +\# \text{ initial }U\!D\text{s,} & \text{\hspace*{15mm}otherwise.}
\end{cases}
$ \\[5mm]
\# terminal 0s & $\leftrightarrow$ & \# terminal $U\!D$s \\[2mm]

length of segment starting \emph{after} & \raisebox{-1.5ex}{$\leftrightarrow$} &   \raisebox{-1.5ex}{\# initial $U\!D$s} \\[-3mm]
the second 0 and ending at 1 &  &     \\[1mm]
    \hline
  \end{tabular}  \\[5mm]
  Table 1
\end{center}

In the last statistic, the segment length is interpreted as 0
if either there is no second 0 or no 1 after the second 0; recall that there is at most a single 1 after a second 0.

Table 1 shows that the three statistics on the left have the same joint distribution on avoiders as the corresponding three on the right have on Dyck paths. But there is a larger family of statistics on avoiders that appear to have the same joint distribution as a corresponding family on Dyck paths, as shown in Table 2. Here, the all-0s and the $012\cdots (n-1)$ ascent sequences and the sawtooth, $(U\!D)^n$, and pyramid, $U^n D^n$, Dyck paths are excluded to simplify the definition of statistics, just as in Table 1. (The reader may give suitable interpretations of the statistics, where needed, in the excluded cases.)

\begin{center}
  \begin{tabular}{@{} ccc @{}}
    \hline
 (0011,0021)-avoiders &  & Dyck paths   \\[-1mm]
 not all 0s, not $012\cdots $ &  & not $(U\!D)^n$, not $U^n D^n$   \\[.5mm]
    \hline
    & & \\[-3mm]
\# 0s & $\leftrightarrow$ & length first ascent \\[1mm]

length of segment starting \emph{at} & \raisebox{-1.5ex}{$\leftrightarrow$} &   \raisebox{-1.5ex}{length first descent} \\[-3mm]
the second 0 and ending at 1 &  &     \\[1mm]

\# terminal 0s & $\leftrightarrow$ & degree of elevation \\[1mm]

minimum repeated entry & $\leftrightarrow$ & \# initial $U\!D$s \\[1mm]

\# terminal ``max possible'' entries & $\leftrightarrow$ & \# terminal $U\!D$s \\[1mm]

\# left to right maxima & $\leftrightarrow$ & \# $U\!D$s (peaks) \\[1mm]

\# right to left minima & $\leftrightarrow$ & \# returns to \gl \\[2mm]

    \hline
  \end{tabular}  \\[5mm]
  Table 2
\end{center}

The second statistic on the left is interpreted as 1
if either there is no second 0 or no 1 after the second 0.
In the fourth statistic, there always is a repeated entry (for $n\ge 2$) because $012 \cdots$ is the only ascent sequence without repeated entries.
In the fifth statistic, a ``max possible'' entry $u_i$ is one for which
the defining inequality ``$u_i \le 1 + \#$ ascents in $(u_1,\dots,u_{i-1})$''
becomes an equality.  We refer the reader to \cite{Ca} for any definitions not given here regarding Dyck paths

It would be interesting to find a bijection to verify the conjectured equality of
joint distributions of Table 2.

\section{A refinement of the case 021}

In this section, we consider a refinement of the set $\mathcal{A}_{n,m}=\mathcal{S}_{n,m}(021)$ and calculate the generating function of the distribution on $\mathcal{A}_{n,m}$ for the statistic recording the largest letter and hence obtain an extension of \cite[Theorem 2.15]{DS}.
Our methods are algebraic and hence provide an alternate proof of that result which was shown bijectively.

Given $n \geq 1$ and $0 \leq m <n$, let $\mathcal{A}_{n,m,r,s}$ denote the subset of $\mathcal{A}_{n,m}$ whose members have largest letter $r$ and last letter $s$, where $0 \leq s \leq r \leq m$.  For example, we have $\pi=0101103030 \in \mathcal{A}_{10,4,3,0}$.  The numbers $a_{n,m,r,s}=|\mathcal{A}_{n,m,r,s}|$ may be determined as described in the following lemma.

\begin{lemma}\label{l0}
The array $a_{n,m,r,s}$ may assume nonzero values only when $n \geq 1$ and $0 \leq s \leq r \leq m <n$.  It satisfies the condition $a_{n,0,0,0}=1$ if $n \geq1$, together with $a_{n,m,0,0}=0$ and $a_{n,m,1,1}=\binom{n-1}{2m-1}$ if $n,m \geq 1$.  If $n \geq 3$ and $1 \leq m \leq n-1$, then the numbers $a_{n,m,r,s}$ satisfy the recurrences
\begin{equation}\label{l0e1}
a_{n,m,r,0}=\sum_{j=0}^{r}a_{n-1,m,r,j}, \qquad r \geq 1,
\end{equation}
and
\begin{equation}\label{l0e2}
a_{n,m,r,r}=a_{n-1,m,r,r}+\sum_{i=1}^{r-1}\sum_{j=0}^i a_{n-1,m-1,i,j}+\sum_{j=0}^r a_{n-2,m-1,r,j}, \qquad r \geq 2,
\end{equation}
with $a_{n,m,r,s}=0$ if $r>s\geq 1$ and $a_{2,1,1,0}=0$.
\end{lemma}
\begin{proof}
The case when $r=s=0$ is clear from the definitions.  Observe also that when $r \geq 1$, deleting the final $0$ from $\pi \in \mathcal{A}_{n,m,r,0}$ provides a bijection with $\cup_{j=0}^{r-1}\mathcal{A}_{n-1,m,r,j}$, which implies \eqref{l0e1}.  Furthermore, note that $\mathcal{A}_{n,m,r,s}$ is empty if $r>s \geq 1$ since we are to avoid $021$.

So suppose $\pi \in \mathcal{A}_{n,m,r,r}$, where $1 \leq r \leq m <n$.  If $r=1$, then $\pi \in \mathcal{A}_{n,m,1,1}$ is a binary sequence that starts with $0$, ends in $1$, and contains exactly $m$ ascents, which implies $a_{n,m,1,1}=\binom{n-1}{2m-1}$.  So assume $r \geq 2$.  Then there are $a_{n-1,m,r,r}$ members of $\mathcal{A}_{n,m,r,r}$ whose penultimate letter is also $r$.  On the other hand, there are $a_{n-1,m-1,i,j}$ members of $\mathcal{A}_{n,m,r,r}$ in which the letter $r$ occurs once, the second largest letter is $i$, and the penultimate letter is $j$, where $0 \leq j \leq i \leq r-1$ and $i \geq 1$.  Summing over all possible $i$ and $j$ implies that there are $\sum_{i=1}^{r-1}\sum_{j=0}^i a_{n-1,m-1,i,j}$ members of $\mathcal{A}_{n,m,r,r}$ in which the letter $r$ occurs once.  Finally, if the letter $r$ occurs more than once within a member of $\mathcal{A}_{n,m,r,r}$ whose penultimate letter is less than $r$, then that letter must be zero, for otherwise ther!
 e would be an occurrence of $021$ (which may be obtained by taking the left-most occurrences of $0$ and $r$, together with the penultimate letter).  Deleting the last two letters from such members of $\mathcal{A}_{n,m,r,r}$ then provides a bijection with the set $\cup_{j=0}^r\mathcal{A}_{n-2,m-1,r,j}$, which has cardinality $\sum_{j=0}^r a_{n-2,m-1,r,j}$.  Combining the three cases above yields \eqref{l0e2}, which completes the proof.
\end{proof}

If $n>m\geq r \geq 0$, then let $A_{n,m,r}(u)=\sum_{s=0}^r a_{n,m,r,s}u^s$.  From the definitions, note that
$$A_{n,m,0}(u)=\begin{cases} 1, & \text{if} \text{~}\text{~} m = 0; \\ 0, & \text{if} \text{~}\text{~} m \geq 1, \end{cases}$$
and
$$A_{n,m,1}(u)=\binom{n-1}{2m}+\binom{n-1}{2m-1}u, \qquad m \geq 1.$$
If $n>m \geq 0$, then let
$A_{n,m}(u,v)=\sum_{r=0}^m A_{n,m,r}(u)v^r$ and $B_{n,m}(v)=\sum_{r=0}^m a_{n,m,r,r}v^r$.
Note that for all $n \geq 1$, we have $A_{n,0}(u,v)=B_{n,0}(v)=1$, $B_{n,1}(v)=(n-1)v$, and
$$A_{n,1}(u,v)=\binom{n-1}{2}v+(n-1)uv.$$

The polynomials $A_{n,m}(u,v)$ and $B_{n,m}(v)$ satisfy the following recurrences when $m \geq 2$.

\begin{lemma}\label{l1}
If $m \geq 2$, then
\begin{align}
A_{n,m}(u,v)&=B_{n-1,m}(uv)+A_{n-1,m}(1,v)+A_{n-2,m-1}(1,uv)\notag\\
&~~+\frac{uv}{1-uv}(A_{n-1,m-1}(1,uv)-(uv)^mA_{n-1,m-1}(1,1))\label{l1e1}
\end{align}
and
\begin{align}
B_{n,m}(v)&=B_{n-1,m}(v)+A_{n-2,m-1}(1,v)+\frac{v}{1-v}(A_{n-1,m-1}(1,v)-v^mA_{n-1,m-1}(1,1)).\label{l1e2}
\end{align}
\end{lemma}
\begin{proof}
By \eqref{l0e1} and \eqref{l0e2}, we have
$$A_{n,m,r}(u)=A_{n-1,m,r}(1)+u^ra_{n-1,m,r,r}+u^rA_{n-2,m-1,r}(1)+u^r\sum_{i=1}^{r-1}A_{n-1,m-1,i}(1), \qquad r \geq 2.$$
Thus, for $m \geq 2$, we have
\begin{align*}
A_{n,m}(u,v)&=A_{n,m,0}(u)+A_{n,m,1}(u)v+\sum_{r=2}^m A_{n,m,r}(u)v^r\\
&=\binom{n-1}{2m}v+\binom{n-1}{2m-1}uv+(A_{n-1,m}(1,v)-A_{n-1,m,1}(1)v)\\
&~~+(B_{n-1,m}(uv)-a_{n-1,m,1,1}uv)+(A_{n-2,m-1}(1,uv)-A_{n-2,m-1,1}(1)uv)\\
&~~+\sum_{r=2}^m (uv)^r\sum_{i=1}^{r-1}A_{n-1,m-1,i}(1)\\
&=\binom{n-1}{2m}v+\binom{n-1}{2m-1}uv+\left(A_{n-1,m}(1,v)-\binom{n-1}{2m}v\right)\\
&~~+\left(B_{n-1,m}(uv)-\binom{n-2}{2m-1}uv\right)+\left(A_{n-2,m-1}(1,uv)-\binom{n-2}{2m-2}uv\right)\\
&~~+\sum_{i=1}^{m-1}A_{n-1,m-1,i}(1)\left(\frac{(uv)^{i+1}-(uv)^{m+1}}{1-uv}\right)\\
&=B_{n-1,m}(uv)+A_{n-1,m}(1,v)+A_{n-2,m-1}(1,uv)\\
&~~+\frac{uv}{1-uv}(A_{n-1,m-1}(1,uv)-(uv)^mA_{n-1,m-1}(1,1)),
\end{align*}
which gives \eqref{l1e1}.

By \eqref{l0e2}, we also have
\begin{align*}
B_{n,m}(v)&=\binom{n-1}{2m-1}v+\left(B_{n-1,m}(v)-\binom{n-2}{2m-1}v\right)
+\sum_{r=2}^m v^r\sum_{i=1}^{r-1} A_{n-1,m-1,i}(1)\\
&~~+\sum_{r=2}^m A_{n-2,m-1,r}(1)v^r\\
&=\binom{n-2}{2m-2}v+B_{n-1,m}(v)+\sum_{i=1}^{m-1}A_{n-1,m-1,i}(1)\left(\frac{v^{i+1}-v^{m+1}}{1-v}\right)\\
&~~+\left(A_{n-2,m-1}(1,v)-\binom{n-2}{2m-2}v\right)\\
&=B_{n-1,m}(v)+A_{n-2,m-1}(1,v)+\frac{v}{1-v}(A_{n-1,m-1}(1,v)-v^mA_{n-1,m-1}(1,1)),
\end{align*}
which completes the proof.
\end{proof}

If $n \geq 1$, then let
$A_n(u,v,w)=\sum_{m=0}^{n-1} A_{n,m}(u,v)w^m$ and $B_n(v,w)=\sum_{m=0}^{n-1} B_{n,m}(v)w^m$.  Define the generating functions $f(x;u,v,w)$ and $g(x;v,w)$ by
$$f(x;u,v,w)=\sum_{n \geq 1} A_n(u,v,w)x^n$$
and
$$g(x;v,w)=\sum_{n \geq 1} B_n(v,w)x^n.$$
Note that
$$f(x;1,1,w)=\sum_{n \geq 1} A_n(1,1,w)x^n=\sum_{0\leq m<n}S_{n,m}(021)x^nw^m.$$  The following functional equations are satisfied by $f$ and $g$.

\begin{lemma}\label{l2}
We have
\begin{align}
f(x;u,v,w)&=\frac{x-2x^2-wx^3}{1-x}+xg(x;uv,w)+xf(x;1,v,w)\notag\\
&~~+\left(wx^2+\frac{uvwx}{1-uv}\right)f(x;1,uv,w)-\frac{u^2v^2wx}{1-uv}f(x;1,1,uvw)\label{l2e1}
\end{align}
and
\begin{align}
(1-x)g(x;v,w)&=\frac{x-x^2-wx^3}{1-x}+\left(wx^2+\frac{vwx}{1-v}\right)f(x;1,v,w)-\frac{v^2wx}{1-v}f(x;1,1,vw).\label{l2e2}
\end{align}
\end{lemma}
\begin{proof}
Multiplying \eqref{l1e1} by $w^m$ and summing over $m \geq 2$ implies for all $n \geq 3$,
\begin{align}
A_n(u,v,w)&=1+\left(\binom{n-1}{2}v+(n-1)uv\right)w+\left(B_{n-1}(uv,w)-1-(n-2)uvw\right)\notag\\
&~~+\left(A_{n-1}(1,v,w)-1-\binom{n-1}{2}vw\right)+w\left(A_{n-2}(1,uv,w)-1\right)\notag\\
&~~+\frac{uvw}{1-uv}\left(A_{n-1}(1,uv,w)-1\right)-\frac{u^2v^2w}{1-uv}\left(A_{n-1}(1,1,uvw)-1\right)\notag\\
&=-1+B_{n-1}(uv,w)+A_{n-1}(1,v,w)+w\left(A_{n-2}(1,uv,w)-1\right)\notag\\
&~~+\frac{uvw}{1-uv}\left(A_{n-1}(1,uv,w)-uvA_{n-1}(1,1,uvw)\right). \label{l2e3}
\end{align}
Note that \eqref{l2e2} is also seen to hold when $n=2$, upon defining $A_0(u,v,w)=1$.  Similarly, recurrence \eqref{l1e2} implies for all $n \geq 2$,
\begin{align}
B_n(v,w)&=1+(n-1)vw+(B_{n-1}(v,w)-1-(n-2)vw)+w(A_{n-2}(1,v,w)-1)\notag\\
&~~+\frac{vw}{1-v}\left(A_{n-1}(1,v,w)-1\right)-\frac{v^2w}{1-v}\left(A_{n-1}(1,1,vw)-1\right)\notag\\
&=B_{n-1}(v,w)+w(A_{n-2}(1,v,w)-1)+\frac{vw}{1-v}(A_{n-1}(1,v,w)-vA_{n-1}(1,1,vw)). \label{l2e4}
\end{align}
Multiplying \eqref{l2e3} by $x^n$ and summing over $n\geq2$ yields
\begin{align*}
f(x;u,v,w)-x&=-\frac{x^2}{1-x}+xg(x;uv,w)+xf(x;1,v,w)+wx^2\left(f(x;1,uv,w)-\frac{x}{1-x}\right)\\
&~~+\frac{uvwx}{1-uv}\left(f(x;1,uv,w)-uvf(x;1,1,uvw)\right),
\end{align*}
which gives \eqref{l2e1}.
Similarly, \eqref{l2e4} implies
\begin{align*}
g(x;v,w)-x&=xg(x;v,w)+wx^2\left(f(x;1,v,w)-\frac{x}{1-x}\right)\\
&~~+\frac{vwx}{1-v}\left(f(x;1,v,w)-vf(x;1,1,vw)\right),
\end{align*}
which gives \eqref{l2e2}.
\end{proof}

Let $$h(x;t)=\sum_{0 \leq m <n} N_{n,m+1}x^nt^m=\frac{1-x-xt-\sqrt{(1-x-xt)^2-4x^2t}}{2xt}.$$

We now state the generating function of the joint distribution on $\mathcal{S}_n(021)$ for the statistics recording the largest letter and the number of ascents.

\begin{theorem}\label{t3}
We have
$$f(x;1,v,w)=\frac{x(1-v)[(1-x)^2-wx^2]-v^2wx(1-x)h(x;vw)}{(1-x)((1-v)[(1-x)^2-wx^2]-vwx)}.$$
\end{theorem}
\begin{proof}
Letting $u=1$ in \eqref{l2e1}, and substituting the expression for $g(x;v,w)$ in \eqref{l2e2}, implies
\begin{align*}
&\left(1-x-wx^2-\frac{vwx}{1-v}\right)f(x;1,v,w)\\
&\qquad~~=\frac{x-2x^2-wx^3}{1-x}-\frac{v^2wx}{1-v}f(x;1,1,vw)\\
&\qquad~~~~+\frac{x}{1-x}\left(\frac{x-x^2-wx^3}{1-x}+\left(wx^2+\frac{vwx}{1-v}\right)f(x;1,v,w)-\frac{v^2wx}{1-v}f(x;1,1,vw)\right),
\end{align*}
which may be rewritten as
\begin{equation}\label{p3e1}
\left(1-x-\frac{wx^2}{1-x}-\frac{vwx}{(1-v)(1-x)}\right)f(x;1,v,w)=x-\frac{wx^3}{(1-x)^2}-\frac{v^2wx}{(1-v)(1-x)}f(x;1,1,vw).
\end{equation}

To solve \eqref{p3e1}, we use the \emph{kernel method} (see \cite{BBD}).  Setting the coefficient of $f(x;1,v,w)$ on the left-hand side of \eqref{p3e1} equal to zero, and solving for $w=w_o$ in terms of $x$ and $v$, gives
$$w_o=\frac{(1-x)^2(1-v)}{x(x(1-v)+v)}.$$
Letting $w=w_o$ in \eqref{p3e1} then implies
\begin{align}
f(x;1,1,vw_o)&=\frac{((1-x)^2-w_ox^2)(1-v)}{v^2w_o(1-x)}\notag\\
&=\frac{x((1-x)^2-w_ox^2)(x(1-v)+v)}{v^2(1-x)^3}\notag\\
&=\frac{x}{v(1-x)}.\label{p3e2}
\end{align}

If
$$t=vw_o=\frac{v(1-x)^2(1-v)}{x(x(1-v)+v)},$$
then solving for $v=v_o$ in terms of $x$ and $t$ gives
$$v_o=\frac{1-x-xt+\sqrt{(1-x-xt)^2-4x^2t}}{2(1-x)},$$
where we have chosen the principle root.  Therefore, by \eqref{p3e2}, we have
\begin{align*}
f(x;1,1,t)&=\frac{x}{v_o(1-x)}\\
&=\frac{2x}{1-x-xt+\sqrt{(1-x-xt)^2-4x^2t}}\\
&=\frac{1-x-xt-\sqrt{(1-x-xt)^2-4x^2t}}{2xt}.\\
\end{align*}
The desired formula now follows from \eqref{p3e1}.
\end{proof}

\textbf{Remark:}  Note that $v$ marks the largest letter and $w$ the number of ascents in the preceding formula.  Taking $v=1$ recovers the generating function for the cardinality of $\mathcal{S}_{n,m}(021)$, which is $N_{n,m+1}$.  Using \eqref{l2e1} and \eqref{l2e2}, one may obtain an expression for $f(x;u,v,w)$, where $u$ marks the last letter.


\begin{thebibliography}{10}

\bibitem{BBD}
C. Banderier, M. Bousquet-M{\'e}lou, A. Denise, P. Flajolet, D. Gardy, and D. Gouyou-Beauchamps, Generating functions for generating trees,
(\emph{Formal Power Series and Algebraic Combinatorics}, Barcelona, 1999), {\em Discrete Math.} {\bf 246(1-3)} (2002) 29--55.

\bibitem{BCD}
M. Bousquet-M{\'e}lou, A. Claesson, M. Dukes, and S. Kitaev, (2+2)-free posets, ascent sequences and pattern avoiding permutations, {\em J. Combin. Theory Ser. A} {\bf 117(7)} (2010) 884--909.

\bibitem{Br}
P. Br\"{a}nd\'{e}n, $q$-Narayana numbers and the flag $h$-vector of $J(2 \times n)$, {\em Discrete Math.} {\bf 281(1-3)} (2004) 67--81.

\bibitem{Ca}
D. Callan, Another bijection for 021-avoiding ascent sequences, preprint,
http://front.math.ucdavis.edu/1402.5898, 6 pp.

\bibitem{CD}
W. Y. C. Chen, A. Y. L. Dai, T. Dokos, T. Dwyer, and B. E. Sagan, On 021-avoiding ascent sequences, \emph{Electron. J. Combin.} {\bf20(1)} (2013) \#P76.

\bibitem{DP}
M. Dukes and R. Parviainen, Ascent sequences and upper triangular matrices containing non-negative integers, {\em Electron. J. Combin.} {\bf17(1)} (2010) \#R53.

\bibitem{DRS}
M. Dukes, J. Remmel, S. Kitaev, and E. Steingr\'{\i}msson, Enumerating (2+2)-free posets by indistinguishable elements, {\em J. Comb.} {\bf2(1)} (2011) 139--163.

\bibitem{DS}
P. Duncan and E. Steingr\'{\i}msson, Pattern avoidance in ascent sequences, {\em Electron. J. Combin.} {\bf18(1)} (2011) \#P226.

\bibitem{K}
S. Kitaev, {\em Patterns in Permutations and Words,} occurring in {\em Monographs in Theoretical Computer Science} (with a forward by J. Remmel), Springer-Verlag, ISBN 978-3-642-17332-5, 2011.

\bibitem{KR}
S. Kitaev and J. Remmel, Enumerating (2+2)-free posets by the number of minimal elements and other statistics, {\em Discrete Appl. Math.} {\bf 159} (2011) 2098--2108.

\bibitem{Kl}
M. Klazar, On $abab$-free and $abba$-free set partitions, \emph{European J. Combin.} {\bf17} (1996) 53--68.

\bibitem{Kn}
D. E. Knuth, {\em The Art of Computer Programming, Vol's. $1$ and $3$}, Addison-Wesley, Reading, Mass. 1968, 1973.

\bibitem{MS0}
T. Mansour and M. Shattuck, Pattern-avoiding set partitions and Catalan numbers, \emph{Electron. J. Combin. (The Zeilberger
   Festschrift)} {\bf18(2)} (2011-12) \#P34.

\bibitem{MS}
T. Mansour and M. Shattuck, Some enumerative results related to ascent sequences, \emph{Discrete Math.} {\bf315-316} (2014) 29--41.

\bibitem{Sl}
N. J. Sloane, \emph{On-line Encyclopedia of Integer Sequences}, http://oeis.org, 2010.

\bibitem{St}
R. P. Stanley, http://www-math.mit.edu/rstan/, ``Enumerative Combinatorics--Catalan addendum''.

\bibitem{Y}
S. H. F. Yan, Ascent sequences and 3-nonnesting set partitions, \emph{European J. Combin.} {\bf39} (2014) 80--94.

\end{thebibliography}
\end{document}